\documentclass[12pt]{amsart}

\usepackage{fullpage}

\usepackage{tikz}

\theoremstyle{plain}
\newtheorem{theorem}{Theorem}[section]

\newtheorem{lemma}[theorem]{Lemma}
\newtheorem{proposition}[theorem]{Proposition}

\theoremstyle{definition}

\newtheorem{problem}[theorem]{Problem}

\begin{document}

\title{A variety of Steiner loops satisfying Moufang's theorem: A solution to Rajah's Problem}

\begin{abstract}
A loop $X$ is said to satisfy Moufang's theorem if for every $x,y,z\in X$ such that $x(yz)=(xy)z$ the subloop generated by $x$, $y$, $z$ is a group. We prove that the variety $V$ of Steiner loops satisfying the identity $(xz)(((xy)z)(yz)) = ((xz)((xy)z))(yz)$ is not contained in the variety of Moufang loops, yet every loop in $V$ satisfies Moufang's theorem. This solves a problem posed by Andrew Rajah.
\end{abstract}

\keywords{Moufang's theorem, Steiner loop, Steiner triple system, Pasch configuration}

\subjclass[2010]{20N05, 05B07}

\author{Ale\v s Dr\'apal}
\address[Dr\'apal]{Department of Mathematics, Charles University, Sokolovsk\'a 83, 186 75, Praha 8, Czech Republic}
\email{drapal@karlin.mff.cuni.cz}

\author{Petr Vojt\v echovsk\' y}
\address[Vojt\v{e}chovsk\'y]{Department of Mathematics, University of Denver, 2390 S.~York St, Denver, Colorado, 80208, USA}
\email{petr@math.du.edu}

\maketitle

\section{Introduction}

A \emph{Moufang loop} is a loop satisfying the identity $x(y(xz)) = ((xy)x)z$. Moufang's theorem \cite{Moufang} states that in every Moufang loop $X$ the implication
\begin{equation}\label{Eq:Rajah}
    x(yz) = (xy)z\quad\Longrightarrow\quad\langle x, y, z\rangle \text{ is a group}
\end{equation}
holds for every $x,y,z\in X$. A short proof of Moufang's theorem can be found in \cite{Drapal}.

We say that a loop $X$ \emph{satisfies Moufang's theorem} if \eqref{Eq:Rajah} holds for every $x,y,z\in X$, and a class $\mathcal X$ of loops \emph{satisfies Moufang's theorem} if every $X\in\mathcal X$ satisfies Moufang's theorem.

In 2011, Andrew Rajah asked \cite{web}: \emph{Is there a variety of loops not contained in the variety of Moufang loops that satisfies Moufang's theorem?} In this short note we give an affirmative answer to Rajah's question. 

Some \emph{classes} of non-Moufang loops satisfying Moufang's theorem are known. In \cite{CEtAl}, Colbourn et al observed that a Steiner loop satisfies Moufang's theorem if and only if its corresponding Steiner triple system has the property that every Pasch configuration generates a subsystem of order $7$, i.e., a Fano plane. They also determined the spectrum of finite orders for which there exist non-Moufang Steiner loops satisfying Moufang's theorem. In \cite{Stuhl}, Stuhl proved that all oriented Hall loops of exponent $4$ satisfy Moufang's theorem. Our solution follows a similar line of reasoning. Marina Rasskazova recently announced an independent solution \cite{Rasskazova}.

Recall that a \emph{Steiner loop} \cite{ColbournRosa} is a loop satisfying the identities
\begin{equation}\label{Eq:Steiner}
    xy=yx\quad\text{and}\quad x(xy)=y.
\end{equation}
A \emph{Steiner quasigroup} is a quasigroup satisfying the identities
\begin{equation}\label{Eq:SQ}
    xy=yx,\quad x(xy)=y\quad\text{and}\quad xx=x.
\end{equation}
There is a one-to-one correspondence between Steiner loops and Steiner quasigroups. Given a Steiner loop $X$ with identity element $1$, the corresponding Steiner quasigroup is obtained by removing $1$ and setting $xx=x$ for every $x\in X\setminus\{1\}$. Conversely, given a Steiner quasigroup $X$, the corresponding Steiner loop is obtained by adjoining a new element $1$ and setting $x1 = 1x = x$ for every $x\in X\cup\{1\}$.

A \emph{Steiner triple system} is a partition of the edges of a complete graph into edge-disjoint triangles, with the vertices called \emph{points} and the triangles called \emph{blocks}. There is a one-to-one correspondence between Steiner triple systems and Steiner quasigroups. Given a Steiner triple system $S$ on $X$, the corresponding Steiner quasigroup $(X,\cdot)$ is defined by setting $xy=z$ if $\{x,y,z\}$ is a block of $S$, and $xx=x$ for every $x\in X$. Conversely, given a Steiner quasigroup $(X,\cdot)$, the corresponding Steiner triple system $STS(X)$ is obtained by declaring $\{x,y,z\}$ to be a block whenever $x\ne y$ satisfy $xy=z$.

\begin{figure}
\begin{center}
\begin{tikzpicture}
[scale = 2.0, point/.style={circle,fill=gray!20,inner sep=0pt,minimum size=7mm},]
    \node at (0,1.5) {Pasch configuration};
    \node[point] (cc) at (0,0) {$c$};
    \node[point] (bl) at (-0.866,-0.5) {$x$};
    \node[point] (br) at (0.866,-0.5) {$z$};
    \node[point] (cl) at (-0.433,0.25) {$xy$};
    \node[point] (cr) at (0.433,0.25) {$yz$};
    \node[point] (tc) at (0,1) {$y$};
    \draw (tc)--(cl);
    \draw (cl)--(bl);
    \draw (tc)--(cr);
    \draw (cr)--(br);
    \draw (cl)--(cc);
    \draw (cc)--(br);
    \draw (bl)--(cc);
    \draw (cc)--(cr);
    \node at (3,1.5) {Fano plane};
    \node[point] (cc) at (3,0) {$c$};
    \node[point] (bl) at (2.134,-0.5) {$x$};
    \node[point] (bc) at (3,-0.5) {$xz$};
    \node[point] (br) at (3.866,-0.5) {$z$};
    \node[point] (cl) at (2.567,0.25) {$xy$};
    \node[point] (cr) at (3.433,0.25) {$yz$};
    \node[point] (tc) at (3,1) {$y$};
    \draw (cl) to [bend left=38] (cr);
    \draw (cr) to [bend left=38] (bc);
    \draw (bc) to [bend left=38] (cl);
    \draw (bl)--(bc);
    \draw (bc)--(br);
    \draw (tc)--(cl);
    \draw (cl)--(bl);
    \draw (tc)--(cr);
    \draw (cr)--(br);
    \draw (bc)--(cc);
    \draw (cc)--(tc);
    \draw (cl)--(cc);
    \draw (cc)--(br);
    \draw (bl)--(cc);
    \draw (cc)--(cr);
\end{tikzpicture}
\end{center}
\caption{Pasch configuration and Fano plane in a Steiner triple system.}\label{Fg:Pasch}
\end{figure}
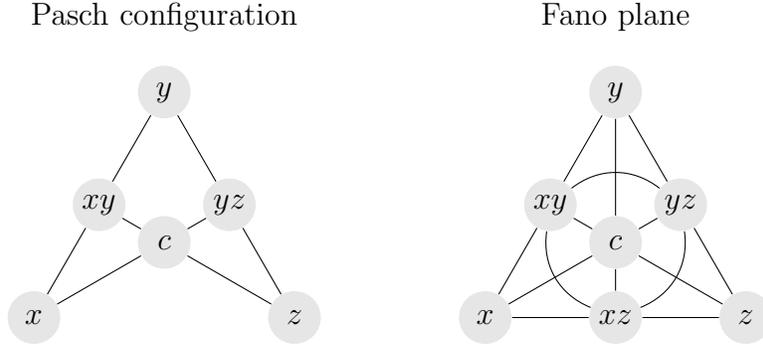

Let $X$ be a Steiner quasigroup and $S=STS(X)$. Suppose that $x,y,z\in X$ are three points not contained in a block of $S$. Then $x(yz)=(xy)z$ holds if and only if the three points
$x$, $y$, $z$ give rise to the \emph{Pasch configuration} in Figure \ref{Fg:Pasch}, with $c=x(yz)=(xy)z$. There is, of course, also the block $\{x,z,xz\}$, which is usually not depicted in a Pasch configuration. If, in addition, $x(yz)=y(xz)$ and $(xy)(yz)=xz$, then the three points $x$, $y$, $z$ give rise to the \emph{Fano plane} in Figure \ref{Fg:Pasch}. Note that both $x(yz)=y(xz)$ and $(xy)(yz)=xz$ can be interpreted as instances of associativity in Steiner quasigroups, namely, $(yz)x = x(yz) = y(xz) = y(zx)$ and $(xy)(yz) = xz = x(y(yz))$.

Consider now the corresponding Steiner loop on $X\cup\{1\}$. If $x,y\in X$ are such that $x\ne y$ then $\langle x,y\rangle$ is a Klein group. If $x,y,z\in X$ are three distinct points that give rise to a Fano plane then $\langle x, y, z\rangle$ is an elementary abelian $2$-group of order $8$.


\section{Solution to Rajah's problem}

\begin{proposition}\label{Pr:Fano}
A Steiner loop $X$ satisfies Moufang's theorem if and only if
\begin{equation}\label{Eq:Good}
    x(yz) = (xy)z\quad\Longrightarrow\quad x(yz) = y(xz)
\end{equation}
holds for every $x,y,z\in X$.
\end{proposition}
\begin{proof}
If $X$ is a Steiner loop satisfying Moufang's theorem and $x(yz)=(xy)z$ for some $x,y,z\in X$, then $\langle x,y,z\rangle$ is a commutative group and hence $x(yz)=y(xz)$. Conversely, suppose that $X$ is a Steiner loop satisfying \eqref{Eq:Good} and let $x,y,z\in X$ be such that
\begin{equation}\label{Eq:Ass}
    x(yz)=(xy)z.
\end{equation}
If $1\in\{x,y,z\}$ or $\{x,y,z\}$ is contained in a block then $\langle x,y,z\rangle$ is a group. For the rest of the proof suppose that $x$, $y$, $z$ are distinct non-identity elements not contained in a block so that the three points $x$, $y$, $z$ form a Pasch configuration as in Figure \ref{Fg:Pasch}. By \eqref{Eq:Good}, $x(yz) = y(xz)$ and $\{y,xz,x(yz)\}$ is a block. Furthermore, with $u=x$, $v=xy$ and $w=(xy)z$ we have
\begin{displaymath}
   u(vw) = xz = y(y(xz)) \stackrel{\eqref{Eq:Good}}{=} y(x(yz)) = (x(xy))(x(yz)) \stackrel{\eqref{Eq:Ass}}{=} (x(xy))((xy)z)) = (uv)w.
\end{displaymath}
Applying \eqref{Eq:Good} with $(u,v,w)$ in place of $(x,y,z)$, we obtain
\begin{displaymath}
    xz = u(vw) = v(uw) = (xy)(x((xy)z)) \stackrel{\eqref{Eq:Ass}}{=} (xy)(x(x(yz))) = (xy)(yz).
\end{displaymath}
This shows that $\{xy,yz,xz\}$ is a block and thus $\langle x,y,z\rangle$ is a group.
\end{proof}

\begin{lemma}\label{Lm:Id}
Any Steiner loop satisfying the identity
\begin{equation}\label{Eq:Id}
    (xz)(((xy)z)(yz)) = ((xz)((xy)z))(yz)
\end{equation}
satisfies Moufang's theorem.
\end{lemma}
\begin{proof}
Suppose that $X$ is a Steiner loop satisfying \eqref{Eq:Id}. By Proposition \ref{Pr:Fano}, it suffices to check that \eqref{Eq:Good} holds. Let $x,y,z\in X$ be such that \eqref{Eq:Ass} holds. Then
\begin{align*}
    z &= (xz)x = (xz)((x(yz))(yz)) \stackrel{\eqref{Eq:Ass}}{=} (xz)(((xy)z)(yz))\\
      &\stackrel{\eqref{Eq:Id}}{=} ((xz)((xy)z))(yz) \stackrel{\eqref{Eq:Ass}}{=} ((xz)(x(yz)))(yz).
\end{align*}
Multiplying by $yz$ then yields $y=(xz)(x(yz))$, multiplying further by $xz$ yields $(xz)y = x(yz)$, and we obtain $x(yz)=y(xz)$ by commutativity.
\end{proof}

There is a unique Steiner triple system of order $9$, namely the affine triple system of order $9$ with the corresponding Steiner quasigroup $(\mathbb Z_3\times\mathbb Z_3,\cdot)$, $x\cdot y = -x-y$.

\begin{theorem}
Let $V$ be the variety of Steiner loops satisfying the identity \eqref{Eq:Id}. Then $V$ satisfies Moufang's theorem and it is not contained in the variety of Moufang loops.
\end{theorem}
\begin{proof}
By Lemma \ref{Lm:Id}, every $X\in V$ satisfies Moufang's theorem. It can be checked that the Steiner loop of order $10$ satisfies \eqref{Eq:Id} but is not Moufang.
\end{proof}

Note that the Steiner loop of order $10$ satisfies identities that are not consequences of \eqref{Eq:Steiner} and \eqref{Eq:Id}, for instance the identity $(xy)(y(xz)) = x(y((xy)z))$.

\begin{problem}
Describe the equational theory of the variety of Steiner loops generated by the Steiner loop of order $10$.
\end{problem}

\end{document}